\documentclass[12pt,reqno]{amsart}
\usepackage{amsmath, amsfonts, amsthm, mathtools, fullpage, hyperref, xcolor, enumerate}
\usepackage[capitalise]{cleveref}
\usepackage[msc-links, alphabetic]{amsrefs}

\DeclareMathOperator{\diam}{diam}

\newtheorem{thm}{Theorem}[section]

\newtheorem{lemma}[thm]{Lemma}
\newtheorem{prop}[thm]{Proposition}
\theoremstyle{remark}

\newcommand\restr[2]{{
	\left.\kern-\nulldelimiterspace 
	#1 
	\littletaller 
	\right|_{#2} 
	}}

\newcommand{\littletaller}{\mathchoice{\vphantom{\big|}}{}{}{}}

\title{Higher-order Cheeger inequalities for graphons}
\author{Mugdha Mahesh Pokharanakar}
\address{Department of Mathematics, 
Indian Institute of Science Education and Research Bhopal, 
Bhopal Bypass Road, Bhauri, Bhopal 462066, Madhya Pradesh, India}
\curraddr{}
\email{mugdha22@iiserb.ac.in}
\thanks{}


\begin{document}

\begin{abstract}
	The higher-order Cheeger inequalities were established for graphs 
	by Lee, Oveis Gharan and Trevisan.
	We prove analogous inequalities for graphons in this article.
\end{abstract}

\maketitle

\section{Introduction}

\subsection{Motivation}

Cheeger \cite{Cheeger-manifolds70} and Buser \cite{Buser-manifolds82} established 
the inequalities that relate the smallest positive eigenvalue of 
the Laplace--Beltrami operator on a compact Riemannian manifold to its isoperimetric constant. 
These inequalities are together known as the Cheeger--Buser inequality.
The discrete analogs of this inequality for finite graphs were proved by 
Dodziuk \cite{Dodziuk84}, Tanner \cite{Tanner-vertexCheegerEasy}, 
Alon--Milman \cite{Alon-Milman85} and Alon \cite{Alon-vertexCheeger86}.
They give a relation between the smallest positive eigenvalue of 
the Laplacian of a graph and its Cheeger constant.
Lawler and Sokal \cite{Lawler-Sokal-MarkovCheeger88} established 
similar inequality for Markov chains.

Graphons are symmetric Lebesgue measurable functions from $[0,1]^2$ to $[0,1]$.
They were introduced and studied by Lov\'asz, Borgs, Chayes, S\'os, Szegedy, Vesztergombi, 
and others \cites{Lovasz-Szegedy-limitsofDenseGraphs06, Lovasz-etal-Counthomo06, 
Lovasz-etal-cgtseqDenseGraphs08} as limits of sequences of finite graphs.
Recently, Khetan and Mj \cite{Abhishek-Mahan24} established an analog of 
the Cheeger--Buser inequality for graphons. 
They proved that for any connected graphon $W$, the inequality
\begin{equation*}
    \frac{h_W^2}{8} \leq \lambda_W \leq 2 h_W
\end{equation*}
holds, where $h_W$ and $\lambda_W$ denote the Cheeger constant of $W$ 
and the bottom of the spectrum of the Laplacian of $W$, respectively.

Higher-order analogs of the Cheeger--Buser inequality, 
known as the higher-order Cheeger inequalities, were conjectured by 
Miclo \cite{Miclo-HOCIconj} for finite graphs, and he proved those for trees.
More details about this conjecture and its proof in the case of cycles were given by
Daneshgar, Javadi and Miclo \cite{Daneshgar-etal-HOCIconj12}.
Lee, Oveis Gharan and Trevisan \cite{Higher-Cheeger} established the conjecture,
proving that for every finite connected graph $G = (V,E)$ and 
$1 \leq k \leq \lvert V \rvert$, the inequality 
\[
	\frac{\lambda_k}{2} \leq \rho_G(k) \leq O(k^2) \sqrt{\lambda_k}
\]
holds, where $\lambda_k$ is the $k$-th smallest eigenvalue of the Laplacian of $G$,
and $\rho_G(k)$ is the $k$-way expansion constant of $G$. 
Miclo established the higher-order Cheeger inequalities for self-adjoint Markov operators
\cite[Proposition 5]{Miclo-MarkovSpectrum15} and 
compact Riemannian manifolds \cite[Theorem 7]{Miclo-MarkovSpectrum15}. 

\subsection{Results obtained}

In this article, we prove the following theorem 
establishing higher-order Cheeger inequalities for graphons.
\begin{thm} \label{thm:HOCI-graphons}
	Given a connected graphon $W$, the inequality 
	\[
		\frac{\lambda_k}{2} \leq h_W(k) \leq O(k^{3.5}) \sqrt{\lambda_k} 
	\]
	holds, for every positive integer $k$. 
\end{thm}
Here $h_W(k)$ denotes the $k$-way expansion constant of the graphon $W$,
which is an analog of the $k$-way expansion constant for graphs (as in \cite{Higher-Cheeger}) 
and of the $k$-th isoperimetric constant for Markov operators 
(as in\cites{Wang-MarkovSpectralGap14,Miclo-MarkovSpectrum15}), 
and $\lambda_k$ is an analog of the $k$-th smallest eigenvalue of the Laplacian of a graph.
Similar analogs of the Laplacian eigenvalues in case of infinite dimensional Hilbert spaces
have been studied by Wang \cite{Wang-MarkovSpectralGap14} and 
Miclo \cite{Miclo-MarkovSpectrum15}.

Miclo obtained the higher-order Cheeger inequalities for self-adjoint Markov operators 
\cite[Proposition 5]{Miclo-MarkovSpectrum15} 
as a consequence of \cite[Theorem 3.8]{Higher-Cheeger} 
and an approximation procedure \cite[Section 3]{Miclo-MarkovSpectrum15}.
Wang provided a more detailed argument in the proof of 
\cite[Lemma 2.2]{Wang-MarkovSpectralGap14}.
Using \cite[Proposition 5]{Miclo-MarkovSpectrum15}, it follows that 
for any connected graphon $W$, the inequality 
\[
	\frac{\lambda_k}{2} \leq h_W(k) \leq O(k^{4}) \sqrt{\lambda_k} 
\]
holds, for every positive integer $k$.
Our proof of \cref{thm:HOCI-graphons} is motivated by the argument 
used in the proof of \cite[Theorem 1.5]{Higher-Cheeger},
and by the technique used by Trevisan \cite{Trevisan-notes-expanders} to prove 
the higher-order Cheeger inequalities for graphs 
with weaker upper bounds for the expansion constants.

\subsection{Outline of the proof}

Henceforth, in this section, a connected graphon $W$ and 
a positive integer $k$ are fixed, and $I$ denotes the interval $[0,1]$.
We provide the definition of the Laplacian of a graphon in \cref{section:prelim}, 
and the definitions of the quantities $h_W(k)$ and $\lambda_k$ 
are given in \cref{section:HOCI-graphons}. 
Now we outline the proof of \cref{thm:HOCI-graphons}. 
We refer to \cref{section:HOBuser-graphons} and 
\cref{section:HOCheeger-graphons} for the details. 

In \cref{section:HOBuser-graphons}, 
we obtain the lower bound on $h_W(k)$ stated in \cref{thm:HOCI-graphons}.
For this, we observe that the indicator functions of 
any $k$ pairwise disjoint subsets of $I$ each having positive measure, 
form a $k$-dimensional subspace of the $L^2$ space.
Using the fact that these indicator functions are disjointly supported, 
we show that the Rayleigh quotients of elements of this subspace 
with respect to the Laplacian of $W$ are \emph{small enough}.

The desired upper bound on $h_W(k)$ is established in \cref{section:HOCheeger-graphons}.
\cref{lem:supp-subset-small-exp} shows that the support of every nonzero function in $L^2$
contains a set having positive measure and \emph{small} expansion as compared to
the Rayleigh quotient of that function with respect to the Laplacian of $W$.
Let $f_1, \dots, f_k$ be orthonormal functions in $L^2$, and write $F = (f_1, \dots, f_k)$. 
To find $k$ pairwise disjoint subsets of $I$ 
having positive measures and \emph{small} expansions, 
it suffices to find $k$ disjointly supported functions in $L^2$ which have 
\emph{small} Rayleigh quotients as compared to the Rayleigh quotients of $f_1, \dots, f_k$.
It has been shown in \cref{prop:dis-supp-O(k7)} that the ``smoothened'' indicator functions
of $k$ measurable subsets of $I$ which have \emph{reasonably large} mass and 
which are ``well-separated'' (according to the \emph{radial projection distance} 
defined in \cref{section:HOCheeger-graphons}) from each other, 
scaled by $\lVert F \rVert$, serve the purpose.
The existence of such subsets of $I$ is shown in \cref{lem:large-separated}.
They are obtained by merging together, measurable subsets of $I$ which are ``well-separated'',
which have \emph{large} total mass, but each of which has \emph{small} mass. 
To construct such sets, first we prove in \cref{lem:wellspread} that 
the function $F$ is ``spread'' across all directions in $\mathbb{R}^k$.
In \cref{prop:small-separated}, we show that there is a random partition of $\mathbb{R}^k$ 
having \emph{clusters} of \emph{small} diameters such that 
the inverse images of the intersections of their \emph{cores} with the unit sphere, 
under the \emph{radial projection} of $F$ onto the unit sphere, 
are ``well-separated'' and have \emph{large} total mass. 
Further, their \emph{small} diameters guarantee that each of them has \emph{small} mass.

In \cite{Higher-Cheeger}, the higher-order Cheeger inequalities are established for graphs. 
A weaker version of these inequalities is established by Trevisan 
using a simpler argument in \cite{Trevisan-notes-expanders}.
The proof of \cref{thm:HOCI-graphons} is motivated by the arguments 
used in proving \cite[Theorem 1.5]{Higher-Cheeger} and 
in \cite[Chapters 7, 8]{Trevisan-notes-expanders}.

\section{Preliminaries} \label{section:prelim}

A \emph{graphon} $W$ is a Lebesgue measurable function $W \colon [0,1]^2 \to [0,1]$ 
such that $W(x,y) = W(y,x)$ holds for all $x,y \in [0,1]$. 
Henceforth, we will denote the interval $[0,1]$ by $I$, 
and the Lebesgue measure on $I$ by $\mu_L$. 
In the following, a measurable subset (resp. measurable function) means 
a Lebesgue measurable subset (resp. Lebesgue measurable function). 
Given a graphon $W$, the measures $\nu$ and $\eta$ on the $\sigma$-algebras 
consisting of the Lebesgue measurable subsets of $I$ and 
$E \coloneq \{(x,y) \in I^2 \mid y > x\}$, respectively, are defined by 
\[
	\nu(A) = \int_{A \times I} W \quad \text{and} \quad \eta(S) = \int_S W, 
\]	 
for every measurable subset $A$ of $I$ and measurable subset $S$ of $E$. 
Denote the inner products on the Hilbert spaces $L^2(I,\nu)$ and $L^2(E,\eta)$ by
$\langle \cdot,\cdot \rangle_v$ and $\langle \cdot,\cdot \rangle_e$, respectively, 
and the corresponding norms by $\lVert \cdot \rVert_v$ and 
$\lVert \cdot \rVert_e$, respectively. 
Given any $f,g \in L^2(I,\nu)$, and $\phi, \psi \in L^2(E,\eta)$, we have
\[
	\langle f,g\rangle_v 
	= \int_{I} \int_{I} f(x)g(x)W(x,y)\, \mathrm{d}y\, \mathrm{d}x,
\]
and
\[
	\langle \phi,\psi\rangle_e 
	= \int_{0}^{1} \int_{x}^{1} \phi(x,y)\psi(x,y)W(x,y)\, \mathrm{d}y\, \mathrm{d}x.
\]
Further, for any measurable subset $S$ of $I^2$, we define
\[
	\eta(S) = \int_{S} W(x,y)\, \mathrm{d}x\, \mathrm{d}y.
\]
Given any point $x \in I$, the \emph{degree} of $x$ is denoted by $d_W(x)$, 
and is defined as $d_W(x) = \int_{I} W(x,y)\, \mathrm{d}y$.

Khetan and Mj \cite{Abhishek-Mahan24} proved that the map 
$d \colon L^2(I,\nu) \to L^2(E,\eta)$, defined by $df(x,y) = f(y) - f(x)$, 
for every $f \in L^2(I,\nu)$ and $(x,y) \in E$, is a bounded linear operator. 
The bounded, self-adjoint, positive semidefinite linear operator $d^*d$ on $L^2(I,\nu)$ 
is called the \emph{Laplacian} of the graphon $W$, which is denoted by $\Delta_W$. 
Define the linear operator $T_W \colon L^2(I,\nu) \to L^2(I,\nu)$ as 
\[
	(T_W f)(x) = \int_{I} W(x,y)f(y)\, \mathrm{d}y,
\] 
for each $f \in L^2(I,\nu)$ and $x \in I$. Then, we have
\[
	(\Delta_W f)(x) = f(x) - \frac{1}{d_W(x)} (T_Wf)(x),
\] 
for all $f \in L^2(I,\nu)$ and $x \in I$ with positive degree.
If $W$ is a \emph{connected} graphon, that is, 
for any measurable subset $A$ of $I$ with $0 < \mu_L(A) < 1$, 
the integral $\int_{A \times A^{c}} W$ is positive, 
then the function $d_W$ is positive $\mu_L$-a.e.~on $I$.  

Let $W$ be a graphon, and $f$ be a nonzero element of $L^2(I,\nu)$. 
The \emph{Rayleigh quotient} $R_{\Delta_W}(f)$ of $f$ 
with respect to the Laplacian $\Delta_W$ of $W$ is defined as
\[
	R_{\Delta_W}(f) = \frac{\langle \Delta_W f,f\rangle_v}{\langle f,f\rangle_v}.
\] 
Since $W$ is symmetric, it follows that
\begin{equation} \label{eq:num-of-R_DeltaW}
	\langle \Delta_W f,f\rangle_v = \lVert df \rVert_e^2 
	= \frac{1}{2} \int_{I} \int_{I} (f(x) - f(y))^2 W(x,y)\, \mathrm{d}x\, \mathrm{d}y
\end{equation}
holds (see the proof of Lemma 2.2 in \cite{Mugdha-duaCBGraphons25} for more details). 
So, we have 
\[
	R_{\Delta_W}(f) 
	= \frac{\int_{I} \int_{I} (f(x) - f(y))^2 W(x,y)\, \mathrm{d}x\, \mathrm{d}y}
	{2 \int_{I} f(x)^2 d_W(x)\, \mathrm{d}x}.
\] 

\section{Proof of \cref{thm:HOCI-graphons}} \label{section:HOCI-graphons}

Let $k$ be a positive integer. 
Denote by $\mathcal{P}_k$, the set of all $k$-tuples of pairwise disjoint 
measurable subsets of $I$ each having positive measure.
Let $W$ be a connected graphon. 
The \emph{$k$-way expansion constant} of $W$, denoted by $h_W(k)$, is defined as
\[
	h_W(k) = \inf_{(A_1, \dots, A_k) \in \mathcal{P}_k} \max_{1 \leq i \leq k} h_W(A_i),
\] 
where for any measurable subset $A$ of $I$ having positive measure 
\[	
	h_W(A) \coloneq \frac{\eta(A \times A^c)}{\nu(A)}.
\]
Note that $h_W(1) = 0$, as $h_W(I) = 0$, and that 
$h_W(2)$ is same as the Cheeger constant defined by Khetan and Mj \cite{Abhishek-Mahan24}.

Denote the set of all $k$-dimensional subspaces of $L^2(I,\nu)$ by $\mathcal{S}_k$.
Define
\[
	\lambda_k 
	= \inf_{S \in \mathcal{S}_k} \max_{\substack{f \in S \\ \lVert f \rVert_v = 1}} 
	\left\langle \Delta_W f, f \right\rangle _v.	
\]

Observe that $\lambda_1 = 0$, 
which is the smallest eigenvalue of $\Delta_W$ with a constant eigenfunction, 
and that $\lambda_2$ is the bottom of the spectrum 
(see \cite[Section 3.3]{Abhishek-Mahan24}) of the graphon $W$.

Now we provide a proof of \cref{thm:HOCI-graphons}.

\begin{proof}[Proof of \cref{thm:HOCI-graphons}]
	The stated lower bound on $h_W(k)$ is proved in \cref{lem:buser}. 
	Let us prove the upper bound on $h_W(k)$ assuming \cref{prop:dis-supp-O(k7)}
	and \cref{lem:supp-subset-small-exp}. 
	Let $S$ be any $k$-dimensional subspace of $L^2(I,\nu)$. 
	Choose an orthonormal basis $\{f_1, \dots, f_k\}$ of $S$, and 
	disjointly supported nonzero functions $g_1, \dots, g_k \in L^2(I,\nu)$ 
 	satisfying the inequality
	\[
		R_{\Delta_W}(g_i) \leq O(k^7) \max_{1 \leq j \leq k} R_{\Delta_W}(f_j),
	\]
	for all $1 \leq i \leq k$, as guaranteed by \cref{prop:dis-supp-O(k7)}. 
	Now, for each $i$, using \cref{lem:supp-subset-small-exp}, we get
	a measurable subset $A_i$ of $I$ having positive measure such that
	$A_i$ is a subset of the support of the function $g_i$ and the inequality
	\[
		h_W(A_i) \leq \sqrt{2 R_{\Delta_W}(g_i)} 
		\leq O(k^{3.5}) \sqrt{\max_{1 \leq j \leq k} R_{\Delta_W}(f_j)}
		\leq O(k^{3.5}) \sqrt{\max_{\substack{f \in S \\ \lVert f \rVert_v = 1}} 
		\left\langle \Delta_W f, f \right\rangle _v}
	\]
	holds.
	As the functions $g_1, \dots, g_k$ are disjointly supported, we conclude that 
	$(A_1, \dots, A_k)$ is an element of $\mathcal{P}_k$. Thus, we have	
	\[
		h_W(k) \leq \max_{1 \leq i \leq k} h_W(A_i) 
		\leq  O(k^{3.5}) \sqrt{\max_{\substack{f \in S \\ \lVert f \rVert_v = 1}} 
		\left\langle \Delta_W f, f \right\rangle _v}.
	\] 
	This proves that the inequality $h_W(k) \leq O(k^{3.5}) \sqrt{\lambda_k}$ follows.
\end{proof}

\section{Lower bound on $h_W(k)$} \label{section:HOBuser-graphons}

\begin{lemma} \label{lem:rayleigh-indicator}
	Let $W$ be a connected graphon, and 
	$A$ be a measurable subset of $I$ with positive measure. 
	Then, we have $R_{\Delta_W}(1_A) = h_W(A)$,
	where $1_A$ denotes the indicator function of $A$ on $I$.
\end{lemma}

\begin{proof}
	Observe that 
	\begin{align*}
		\langle \Delta_W 1_A,1_A\rangle_v
		& =
		\frac{1}{2} \int_{I} \int_{I} (1_A(x) - 1_A(y))^2 W(x,y)\, \mathrm{d}x\, \mathrm{d}y
		\tag{using \cref{eq:num-of-R_DeltaW}} 
		\\
		& =
		\frac{1}{2} \int_{I} \int_{I} (1_A(x) + 1_A(y) - 2 \cdot 1_A(x)1_A(y)) W(x,y)\, 
		\mathrm{d}x\, \mathrm{d}y
		\\
		& =
		\frac{1}{2} \int_{I} \int_{I} (1_A(x)(1 - 1_A(y)) + 1_A(y)(1 - 1_A(x))) W(x,y)\, 
		\mathrm{d}x\, \mathrm{d}y
		\\
		& =
		\frac{1}{2} \int_{I} \int_{I} (1_A(x)1_{A^c}(y) + 1_A(y)1_{A^c}(x)) W(x,y)\, 
		\mathrm{d}x\, \mathrm{d}y
		\\
		& =
		\int_{I} \int_{I} 1_A(x)1_{A^c}(y) W(x,y)\, \mathrm{d}x\, \mathrm{d}y
		\\
		& =
		\eta(A \times A^c),
	\end{align*}
	and that
	\[
		\langle 1_A,1_A\rangle_v 
		= \int_{I} \int_{I} 1_A(x) W(x,y)\, \mathrm{d}y\, \mathrm{d}x
		= \int_{A} \int_{I} W(x,y)\, \mathrm{d}y\, \mathrm{d}x
		= \nu(A).
	\] 
\end{proof}

\begin{lemma} \label{lem:buser}
	Let W be a connected graphon, and $k$ be a positive integer.
	Then, we have $\lambda_k \leq 2 h_W(k)$.
\end{lemma}

\begin{proof}
	Let $(A_1, \dots, A_k)$ be an arbitrary element of $\mathcal{P}_k$, and 
	$S$ be the subspace of $L^2(I,\nu)$ spanned by the indicator functions 
	$1_{A_i}$ of the sets $A_i$, for $1 \leq i \leq k$. 
	Note that $S$ is a $k$-dimensional subspace of $L^2(I,\nu)$. 
	Take any element $f$ of $S$ with $\lVert f \rVert_v = 1$, 
	and for every $1 \leq i \leq k$, choose $f_i$ in the span of $1_{A_i}$ such that 
	$f = \sum_{1 \leq i \leq k} f_i$. 
	Since $f_1, \dots, f_k$ are disjointly supported functions, 
	for any $x \in I$ (resp. $y \in I$), 
	there is an index $j_x$ (resp. $\ell_y)$ such that 
	$f_i (x)$ (resp. $f_i(y))$ is zero for any $i \neq j_x$ (resp. $i \neq \ell_y)$. 
	Now if $j_x = \ell_y$, then
	\[
		(f(x) - f(y))^2 = (f_{j_x}(x) - f_{\ell_y}(y))^2 
		= \sum_{i = 1}^{k} (f_i(x) - f_i(y))^2,
	\]
	and if $j_x \neq \ell_y$, then
	\[
		(f(x) - f(y))^2 = (f_{j_x}(x) - f_{\ell_y}(y))^2
		\leq 2f_{j_x}(x)^2 + 2f_{\ell_y}(y)^2 
		\leq 2 \sum_{i = 1}^{k} (f_i(x) - f_i(y))^2.
	\] 
	Thus, it follows that 
	\begin{align*}
		\langle \Delta_W f,f\rangle_v 
		& = 
		\frac{1}{2} \int_{I} \int_{I} (f(x) - f(y))^2 W(x,y)\, \mathrm{d}x\, \mathrm{d}y
		\\
		& \leq 
		\sum_{i = 1}^{k} \int_{I} \int_{I} (f_i(x) - f_i(y))^2 W(x,y)\, \mathrm{d}x\, \mathrm{d}y
		\\
		& = 
		2 \sum_{i = 1}^{k} \langle \Delta_W f_i,f_i\rangle_v.
	\end{align*} 
	Then, using the facts that $\lVert f \rVert_v = 1$ and that 
	$f_1, \dots, f_k$ are disjointly supported, we get
	\begin{align*}
		\langle \Delta_W f,f\rangle_v 
		& \leq 
		\frac{2 \sum_{i = 1}^{k} \langle \Delta_W f_i,f_i\rangle_v}
		{\sum_{i = 1}^{k} \langle f_i,f_i\rangle_v}
		\\ 
		& \leq 
		2 \max_{\substack{1 \leq i \leq k \\ f_i \neq 0}} 
		\frac{\langle \Delta_W f_i,f_i\rangle_v}{\langle f_i,f_i\rangle_v}
		\\
		& \leq 
		2 \max_{1 \leq i \leq k} R_{\Delta_W}(1_{A_i})
		\\
		& =
		2 \max_{1 \leq i \leq k} h_W(A_i). 
		\tag{using \cref{lem:rayleigh-indicator}}
	\end{align*} 
	Since the above inequality holds for all $f \in S$ with $\lVert f \rVert_v = 1$, we have 
	\[
		\max_{\substack{f \in S \\ \lVert f \rVert_v = 1}} 
		\left\langle \Delta_W f, f \right\rangle _v 
		\leq 2 \max_{1 \leq i \leq k} h_W(A_i),
	\]
	and hence, $\lambda_k \leq 2 \max_{1 \leq i \leq k} h_W(A_i)$.
	The desired inequality follows. 
\end{proof}

\section{Upper bound on $h_W(k)$} \label{section:HOCheeger-graphons}

In this section, we prove \cref{prop:dis-supp-O(k7)} and \cref{lem:supp-subset-small-exp}
following \cite{Higher-Cheeger} and \cite{Trevisan-notes-expanders}.
Henceforth, we fix a connected graphon $W$ and a positive integer $k$.
The Euclidean inner product and the Euclidean norm on $\mathbb{R}^k$ 
is denoted by $\langle \cdot,\cdot\rangle$ and $\lVert \cdot \rVert$, respectively.
For any nonempty subset $R$ of $\mathbb{R}^k$, its diameter $\diam(R)$, is defined as 
$\diam(R) = \sup_{\mathbf{v}, \mathbf{w} \in R} \lVert \mathbf{v} - \mathbf{w} \rVert$.

Given any functions $f_1, \dots, f_k \in L^2(I,\nu)$, a function $F \colon I \to \mathbb{R}^k$ 
is defined by $F(x) = (f_1(x), \dots, f_k(x))$, for all $x \in I$.
Define a function 
$\overline{F} \colon F^{-1}(\mathbb{R}^k \setminus \{\mathbf{0}\}) \to \mathbb{R}^k$ by
\[
	\overline{F}(x) = \frac{1}{\lVert F(x) \rVert} F(x), 
\]
for any $x \in I$ with $F(x) \neq \mathbf{0}$, 
where $\mathbf{0}$ is the zero vector in $\mathbb{R}^k$. 
Further, we define a function $d_F \colon I \times I \to [0, \infty]$ as follows. 
For any $x,y \in I$, set
\[
	d_F(x,y) =
	\begin{cases}
		\left\lVert \overline{F}(x) - \overline{F}(y) \right\rVert		
		& \text{if } F(x), F(y) \neq \mathbf{0},
		\\
		0
		& \text{if } F(x) = F(y) = \mathbf{0},
		\\
		\infty
		& \text{otherwise}.
	\end{cases} 
\]
Note that $F$ is a measurable function, and $d_F$ is an extended pseudo-metric on $I$.
Given any $x \in I$ and a nonempty subset $A$ of $I$, we write $d_F(x,A)$ 
to denote the infimum of the set $\{d_F(x,y) \mid y \in A\}$.

\begin{lemma} \label{lem:dist-to-set-meable}
	Let $A$ be a nonempty subset of $I$. 
	Then, the function $f_A \colon I \to [0,\infty]$, defined by $f_A(x) = d_F(x,A)$, 
	for any $x \in I$, is measurable.
\end{lemma}

\begin{proof}
	Note that 
	\[
		f_A^{-1}(\infty) = 
		\begin{cases}
			F^{-1}(\mathbb{R}^k \setminus \{\mathbf{0}\}) 
			& \text{if } F(x) = \mathbf{0} \text{ for all } x \in A,
			\\ 
			F^{-1}(\mathbf{0})
			& \text{if } F(x) \neq \mathbf{0} \text{ for all } x \in A,
			\\
			\emptyset
			& \text{otherwise}.
		\end{cases}
	\] 
	In each of the above cases, $f_A^{-1}(\infty)$ is a measurable subset of $I$,
	since the function $F$ is measurable. 
	Hence, it suffices to show that the function $\restr{f_A}{S} \colon S \to [0,\infty)$ 
	is measurable, where $S = f_A^{-1}([0,\infty))$. 
	Since the inequality $\lvert d_F(x,A) - d_F(y,A) \rvert \leq d_F(x,y)$ holds
	for all $x,y \in S$, the function $\restr{f_A}{S}$ is continuous, 
	with the sets $S$ and $[0,\infty)$ equipped with the pseudo-metric 
	$\restr{d_F}{S \times S}$ and the usual metric, respectively. 
	Thus, it is enough to prove that open sets in $(S, \restr{d_F}{S \times S})$ 
	are measurable subsets of $S$.
	Henceforth, by abuse of notation, we denote $\restr{f_A}{S}$ by $f_A$,
	and $\restr{d_F}{S \times S}$ by $d_F$.

	Let $x$ be an element of $S$, and $\varepsilon$ be a positive real number. 
	Let $B(x,\varepsilon)$ denote the open ball in $S$ centered at $x$ with radius $\varepsilon$. 
	Observe that
   	\[
   		B(x,\varepsilon) =
		\begin{cases}
			F^{-1}(\mathbf{0}) \cap S
			& \text{if } F(x) = \mathbf{0},
			\\
			\overline{F}^{-1}(B_{\mathbb{R}^k}(\overline{F}(x),\varepsilon)) \cap S
			& \text{if } F(x) \neq \mathbf{0},
		\end{cases}
   	\]
	where $B_{\mathbb{R}^k}(\overline{F}(x),\varepsilon)$ denotes 
	the open ball in $\mathbb{R}^k$ centered at $\overline{F}(x)$ having radius $\varepsilon$.
	Moreover, if $F(x) = \mathbf{0}$, then we have $S = F^{-1}(\mathbf{0})$ or $S = I$, 
	and hence, 
	\[
		B(x,\varepsilon) = F^{-1}(\mathbf{0}) \cap S = F^{-1}(\mathbf{0}).
	\]	 
	If $F(x) \neq \mathbf{0}$, then 
	$S = F^{-1}(\mathbb{R}^k \setminus \{\mathbf{0}\})$ or $S = I$. 
	Since the domain of the function $\overline{F}$ is 
	$F^{-1}(\mathbb{R}^k \setminus \{\mathbf{0}\})$, 
	in either of the cases, we get
	\[
		B(x,\varepsilon) 
		= \overline{F}^{-1}(B_{\mathbb{R}^k}(\overline{F}(x),\varepsilon)) \cap S 
		= \overline{F}^{-1}(B_{\mathbb{R}^k}(\overline{F}(x),\varepsilon)). 
	\]	
	Let $V$ be a nonempty open set in $(S,d_F)$. 
	Then, for every $x \in V$, there is a real number $\varepsilon_x > 0$ such that 
	$B(x,\varepsilon_x) \subseteq V$, and so $V = \cup_{x \in V} B(x,\varepsilon_x)$. 
	Now, if $F(x) = \mathbf{0}$ for all $x \in V$, then $V = F^{-1}(\mathbf{0})$, 
	which is measurable.
	If $F(x) \neq \mathbf{0}$ for all $x \in V$, then 
	\[
		V = \bigcup_{x \in V} 
		\overline{F}^{-1}(B_{\mathbb{R}^k}(\overline{F}(x),\varepsilon_x))
		= \overline{F}^{-1} \left( \bigcup_{x \in V} 
		B_{\mathbb{R}^k}(\overline{F}(x),\varepsilon_x) \right),
	\] 
	which is measurable, being the inverse image of an open set in $\mathbb{R}^k$ 
	under the measurable function $\overline{F}$. 
	Otherwise, we have 
	 \[
		V = F^{-1}(\mathbf{0}) \cup \overline{F}^{-1} 
		\left( \bigcup_{x \in V \cap F^{-1}(\mathbb{R}^k \setminus \{\mathbf{0}\})} 
		B_{\mathbb{R}^k}(\overline{F}(x),\varepsilon_x) \right),
	\] 
	which is also measurable.
\end{proof}

\begin{lemma} \label{lem:wellspread}
	Given any orthonormal subset $\{f_1, \dots, f_k\}$ of $L^2(I,\nu)$ 
	and a unit vector $\mathbf{v} \in \mathbb{R}^k$, we have
	\[
		\int_{I} \langle F(x),\mathbf{v} \rangle^2 d_W(x)\, \mathrm{d}x = 1.
	\] 
\end{lemma}

\begin{proof}
	Consider the map $U \colon \mathbb{R}^k \to L^2(I,\nu)$ defined by 
	$(U \mathbf{v})(x) = \langle F(x),\mathbf{v}\rangle$, 
	for any $\mathbf{v} \in \mathbb{R}^k$ and $x \in I$. 
	This is well-defined, since for every $\mathbf{v} = (v_1, \dots, v_k) \in \mathbb{R}^k$, 
	the function $U \mathbf{v} \colon I \to  \mathbb{R}$ is measurable, and we have
	\begin{align*}
		\langle U \mathbf{v},U \mathbf{v}\rangle_v 
		& = 
		\int_{I} \langle F(x),\mathbf{v} \rangle^2 d_W(x)\, \mathrm{d}x
		\\
		& \leq 
		\lVert \mathbf{v} \rVert^2 \int_{I} \lVert F(x) \rVert^2 d_W(x)\, \mathrm{d}x
		\\
		& =
		\lVert \mathbf{v} \rVert^2 
		\int_{I} \left( \sum_{i=1}^{k} f_i(x)^2 \right) d_W(x)\, \mathrm{d}x
		\\
		& =
		\lVert \mathbf{v} \rVert^2 \sum_{i=1}^{k} \lVert f_i \rVert_v^2
		\\
		& < 
		\infty.
	\end{align*} 
	Note that $U$ is a bounded linear transformation, 
	and the adjoint $U^* \colon L^2(I,\nu) \to \mathbb{R}^k$ of $U$ is given by 
	$U^* f = (\langle f_1,f\rangle_v, \dots, \langle f_k,f\rangle_v)$, 
	for every $f \in L^2(I,\nu)$. 
	Let $\{e_1, \dots, e_k\}$ denote the standard basis of $\mathbb{R}^k$. 
	As the subset $\{f_1, \dots, f_k\}$ of $L^2(I,\nu)$ is orthonormal, 
	we get $U^* Ue_i = U^* f_i = e_i$, for all $1 \leq i \leq k$, and hence, 
	$U^* U$ is the identity operator on $\mathbb{R}^k$. 
	Therefore, for any unit vector $\mathbf{v} \in \mathbb{R}^k$, we obtain
	\[
		\int_{I} \langle F(x),\mathbf{v} \rangle^2 d_W(x)\, \mathrm{d}x 
		= \langle U \mathbf{v},U \mathbf{v}\rangle_v
		= \langle U^* U \mathbf{v},\mathbf{v}\rangle 
		= \langle \mathbf{v},\mathbf{v}\rangle = 1.
	\]
\end{proof}

Let $A$ and $B$ be subsets of $\mathbb{R}$, and $t$ be a real number. 
The sets $A + B$ and $tA$ are defined as
\[
	A + B = \{a + b \mid a \in A, b \in B\} \quad \text{and} \quad
	tA = \{ta \mid a \in A\}.
\] 
We denote the sets $A + \{x\}$ and $A + \{-x\}$ by $A + x$ and $A - x$, respectively.
Henceforth, $\mathcal{B}$ denotes the Borel $\sigma$-algebra on $\mathbb{R}$.

\begin{lemma} \label{lem:fit-in-0tos}
	Let $(\mathbb{R},\mathcal{B},\mu)$ be a measure space, 
	and suppose that $\mu$ is a translation invariant measure.
	Let $s$ be a positive real number, and $A \in \mathcal{B}$ be such that 
	no two distinct elements of $A$ are congruent modulo $s\mathbb{Z}$.
	Then, there exists a Borel measurable subset $B$ of $[0,s)$ satisfying
	$A + s\mathbb{Z} = B + s\mathbb{Z}$ and $\mu(B) = \mu(A)$.
\end{lemma}

\begin{proof}
	For any integer $n$, let $A_n$ denote the set $A \cap [ns,(n+1)s)$. 
	Then the set $B \coloneq \bigcup_{n \in \mathbb{Z}} (A_n - ns)$ is 
	a Borel measurable subset of $[0,s)$ such that 
	$A + s\mathbb{Z} = B + s\mathbb{Z}$ holds.
	Further, since any two distinct elements of $A$ are not congruent modulo $s\mathbb{Z}$,
	if $m \neq n$, then the sets $A_m - ms$ and $A_n - ns$ are disjoint.
	Hence, it follows from translation invariance of the measure $\mu$ that 
	$\mu(B) = \sum_{n \in \mathbb{Z}} \mu(A_n) = \mu(A)$.
\end{proof}

Given any subset $R$ of the unit sphere $\mathbb{S}^{k-1}$ in $\mathbb{R}^k$, 
we denote the set $\overline{F}^{-1}(R)$ by $V(R)$.

\begin{prop} \label{prop:small-separated}
	Let $\{f_1, \dots, f_k\}$ be an orthonormal subset of $L^2(I,\nu)$.
	Then for some $m \geq 1$, there exist pairwise disjoint measurable subsets 
	$T_1, \dots, T_m$ of $I$ such that the following conditions hold. 
	\begin{enumerate}[(a)]	
		\item $\sum_{i=1}^{m} \int_{T_i} \lVert F(x) \rVert^2 d_W(x)\, \mathrm{d}x 
		\geq k - \frac{1}{4}$.
		\item For any $1 \leq i,j \leq k$ with $i \neq j$, 
		if $x \in T_i$ and $y \in T_j$, then  
		$d_F(x,y) \geq \frac{1}{4 \sqrt{5} k^3}$.
		\item $\int_{T_i} \lVert F(x) \rVert^2 d_W(x)\, \mathrm{d}x 
		\leq 1 + \frac{1}{4k}$ for all $1 \leq i \leq m$. 
	\end{enumerate}
\end{prop}

\begin{proof}
	Let us put $s = \frac{1}{\sqrt{5}k}$.
	For any element $\mathbf{n} = (n_1, \dots, n_k)$ of $\mathbb{Z}^k$, let 
	\[
		C_{\mathbf{n}} \coloneq \prod_{i=1}^{k} [n_i s, n_i s + s) 
		\quad \text{and} \quad 
		\widetilde{C}_{\mathbf{n}} \coloneq 
		\prod_{i=1}^{k} \left[ n_i s + \frac{s}{8k^2}, n_i s + s - \frac{s}{8k^2} \right].	
	\]
	We denote the cube $\widetilde{C}_{\mathbf{n}} + \mathbf{w}$ by 
	$\widetilde{C}_{\mathbf{n, w}}$ for any $\mathbf{w} \in \mathbb{R}^k$.
	For any integer $n$, let
	\[
		J_n \coloneq \left[ ns + \frac{s}{8k^2}, ns + s - \frac{s}{8k^2} \right].
	\]
	Let $m$ denote the Borel measure on $\mathbb{R}$. 
	For any $A \in \mathcal{B}$, define $\mu(A) = \frac{1}{s} m(A)$. 
	Then, $\mu$ is a translation invariant measure on $\mathbb{R}$. 
	Its restriction, again denoted by $\mu$, 
	to the Borel measurable subsets of $[0,s)$ is a probability measure. 
	Consider the corresponding product probability space $\Omega = [0,s)^k$ 
	equipped with the Borel $\sigma$-algebra and the probability measure $\mathbb{P}$. 
	Note that
	\[
		\mathbb{P}(I_1 \times \dots \times I_k) 
		= \frac{m(I_1) \times \dots \times m(I_k)}{s^k},
	\] 
	for all Borel measurable subsets $I_1, \dots, I_k$ of $[0,s)$. 
	
	Now fix any $x \in I$ such that $F(x) \neq \mathbf{0}$, 
	and let $\overline{F}(x) = \mathbf{z} = (z_1, \dots, z_k)$. Define
	\[
		A_x = \left\{ \mathbf{w} \in \Omega \biggm| \overline{F}(x) \in 
		\bigcup_{\mathbf{n} \in \mathbb{Z}^k} \widetilde{C}_{\mathbf{n, w}} \right\}.
	\] 
	Given any $\mathbf{w} \in \Omega$, observe that $\mathbf{w} \in A_x$
	if and only if $\overline{F}(x) - \mathbf{w} \in 
	\bigcup_{\mathbf{n} \in \mathbb{Z}^k} \widetilde{C}_{\mathbf{n}}$, that is,
	$\mathbf{w} - \overline{F}(x) \in 
	- \bigcup_{\mathbf{n} \in \mathbb{Z}^k} \widetilde{C}_{\mathbf{n}}$.
	Using the fact that $\bigcup_{\mathbf{n} \in \mathbb{Z}^k} \widetilde{C}_{\mathbf{n}}
	= - \bigcup_{\mathbf{n} \in \mathbb{Z}^k} \widetilde{C}_{\mathbf{n}}$, we get
	\begin{align*}
		A_x 
		& = 
		\Omega \cap \left( \left( \bigcup_{\mathbf{n} \in \mathbb{Z}^k} 
		\widetilde{C}_{\mathbf{n}} \right) + \overline{F}(x) \right)
		\\
		& =
		\prod_{i=1}^{k} \left( [0,s) \cap 
		\left( \bigcup_{n \in \mathbb{Z}} (J_n + z_i) \right) \right) 
		\\
		& =
		\prod_{i=1}^{k} \left( [0,s) \cap ((J_0 + z_i) + s\mathbb{Z}) \right).
	\end{align*}
	Let $i \in \{1, \dots, k\}$ be arbitrary. 
	Since the interval $J_0 + z_i$ has length less than $s$, 
	no two of its elements are congruent modulo $s\mathbb{Z}$.
	Hence, using \cref{lem:fit-in-0tos}, choose a Borel measurable subset $B_i$ 
	of $[0,s)$ satisfying $(J_0 + z_i) + s\mathbb{Z} = B_i + s\mathbb{Z}$, and 
	\[
		\mu(B_i) = \mu(J_0 + z_i) = \mu(J_0) = s - \frac{s}{4k^2}.
	\]
	Then, note that $A_x = \prod_{i=1}^{k} B_i$, and thus, 
	$A_x$ is a Borel measurable subset of $\Omega$ with 
	$\mathbb{P}(A_x) = \left( 1 - \frac{1}{4k^2} \right)^k$. 
	Using Bernoulli's inequality, it follows that $\mathbb{P}(A_x) \geq 1 - \frac{1}{4k}$. 

	For any $\mathbf{w} \in \Omega$, let $R_{\mathbf{w}}$ denote the Borel measurable set 
	$\left( \bigcup_{\mathbf{n} \in \mathbb{Z}^k} \widetilde{C}_{\mathbf{n, w}} \right) 
	\cap \mathbb{S}^{k-1}$. 
	As $\overline{F}$ is a measurable function, the set $V(R_{\mathbf{w}})$ is measurable. 
	For every $x \in I$, define the function $X_x \colon \Omega \to \mathbb{R}$ by
	\[
		X_x(\mathbf{w}) = 
		\begin{cases}
			\lVert F(x) \rVert^2 1_{A_x}(\mathbf{w})
			& \text{if } F(x) \neq \mathbf{0},
			\\
			0
			& \text{otherwise},
		 \end{cases}
	\] 
	for any $\mathbf{w} \in \Omega$. 
	If $x$ is an element if $I$ such that $F(x) = \mathbf{0}$, 
	then $X_x$ is a random variable. 
	Also, for each $x \in I$ with $F(x) \neq \mathbf{0}$, 
	the function $X_x$ is a random variable, 
	as $A_x$ is a Borel measurable subset of $\Omega$. 
	Moreover, since for all $\mathbf{w} \in \Omega$ and $x \in I$, the equality
	$X_x(\mathbf{w}) = \lVert F(x) \rVert^2 1_{V(R_{\mathbf{w}})}(x)$ holds, 
	the function $x \mapsto X_x(\mathbf{w})$ on $I$ is measurable 
	for every $\mathbf{w} \in \Omega$. 
	For any $\mathbf{w} \in \Omega$, define
	\[
		X(\mathbf{w}) = \int_{I} X_x(\mathbf{w}) d_W(x)\, \mathrm{d}x.
	\] 
	Note that 
	\[
		X(\mathbf{w}) 
		= \int_{I} \lVert F(x) \rVert^2 1_{V(R_{\mathbf{w}})}(x) d_W(x)\, \mathrm{d}x,
	\] 
	for all $\mathbf{w} \in \Omega$. 
	We will now show that $X$ is a random variable. 
	We know that the function $(x,\mathbf{w}) \mapsto \lVert F(x) \rVert^2 d_W(x)$ 
	is measurable on $I \times \Omega$. 
	Observe that for any $(x,\mathbf{w}) \in I \times \Omega$, we have 
	$1_{V(R_{\mathbf{w}})}(x) = 1$ if and only if 
	$x \in F^{-1}(\mathbb{R}^k \setminus \{\mathbf{0}\})$ and $\overline{F}(x) \in 
	\bigcup_{\mathbf{n} \in \mathbb{Z}^k} \widetilde{C}_{\mathbf{n}} + \mathbf{w}$ hold.
	Let 
	\[
		H \colon F^{-1}(\mathbb{R}^k \setminus \{\mathbf{0}\}) \times \Omega \to \mathbb{R}^k
	\]	
	be the function defined by $H(x,\mathbf{w}) = \overline{F}(x) - \mathbf{w}$. 
	Then, it follows that $1_{V(R_{\mathbf{w}})}(x) = 1$ if and only if 
	$(x,\mathbf{w}) \in H^{-1}\!\left( \bigcup_{\mathbf{n} \in \mathbb{Z}^k} 
	\widetilde{C}_{\mathbf{n}} \right)$. 
	Since the function $(x,\mathbf{w}) \mapsto (\overline{F}(x),\mathbf{w})$ 
	is measurable on $F^{-1}(\mathbb{R}^k \setminus \{\mathbf{0}\}) \times \Omega$, 
	and the function $(\mathbf{v}_1,\mathbf{v}_2) \mapsto \mathbf{v}_1 - \mathbf{v}_2$ 
	from $\mathbb{R}^k \times \mathbb{R}^k$ to $\mathbb{R}^k$ is continuous, 
	their composition $H$ is measurable. 
	Now the fact that $\bigcup_{\mathbf{n} \in \mathbb{Z}^k} \widetilde{C}_{\mathbf{n}}$ 
	is a Borel measurable subset of $\mathbb{R}^k$ ensures that 
	the function $(x,\mathbf{w}) \mapsto 1_{V(R_{\mathbf{w}})}(x)$ 
	is measurable on $I \times \Omega$. 
	We can conclude that the function 
	$(x,\mathbf{w}) \mapsto \lVert F(x) \rVert^2 1_{V(R_{\mathbf{w}})}(x) d_W(x)$
	is measurable on $I \times \Omega$. 
	By the Fubini--Tonelli theorem, it follows that $X$ is a random variable, 
	and we have
	\begin{align*}
		\mathbb{E}[X] 
		& =
		\int_{\Omega} \left( \int_{I} X_x(\mathbf{w}) d_W(x)\, \mathrm{d}x \right) 
		\mathrm{d}\mathbb{P}(\mathbf{w}) 
		\\
		& =
		\int_{I} \left( \int_{\Omega} X_x(\mathbf{w})\, 
		\mathrm{d}\mathbb{P}(\mathbf{w}) \right) d_W(x)\, \mathrm{d}x
		\\
		& =
		\int_{F^{-1}(\mathbb{R}^k \setminus \{0\})}
		\lVert F(x) \rVert^2 \left( \int_{\Omega} 1_{A_x}(\mathbf{w})\, 
		\mathrm{d}\mathbb{P}(\mathbf{w}) \right) d_W(x)\, \mathrm{d}x
		\\
		& =
		\int_{F^{-1}(\mathbb{R}^k \setminus \{0\})}
		\lVert F(x) \rVert^2 \mathbb{P}(A_x) d_W(x)\, \mathrm{d}x
		\\
		& \geq
		\left( 1 - \frac{1}{4k} \right) \int_{I} \lVert F(x) \rVert^2 d_W(x)\, \mathrm{d}x
		\\
		& =
		\left( 1 - \frac{1}{4k} \right) k 
		\tag{using $\lVert f_i \rVert_v = 1$ for all $1 \leq i \leq k$}
		\\
		& =
		k - \frac{1}{4}.
	\end{align*}
	Choose and fix $\mathbf{w} \in \Omega$ such that $X(\mathbf{w}) \geq k - \frac{1}{4}$.
	Then, $X_x(\mathbf{w})$ is positive for some $x \in I$, and hence, 
	the set $V(R_{\mathbf{w}})$ is nonempty. 
	This implies that $R_{\mathbf{w}}$ is a nonempty set, and therefore, the set 
	\[
		\mathcal{N} = \left\{ \mathbf{n} \in \mathbb{Z}^k \Bigm| 
	 	\widetilde{C}_{\mathbf{n,w}} \cap \mathbb{S}^{k-1} \neq \emptyset \right\}
	\]
	is nonempty. Note that $\mathcal{N}$ is a finite set. 
	Write $\mathcal{N} = \{\mathbf{n}_1, \dots, \mathbf{n}_m\}$.
	For each $1 \leq i \leq m$, define 
	$R_i = \widetilde{C}_{\mathbf{n}_i,\mathbf{w}} \cap \mathbb{S}^{k-1}$,
	and $T_i = V(R_i)$. 
	Since $R_1, \dots, R_m$ are pairwise disjoint Borel measurable subsets 
	of $\mathbb{S}^{k-1}$, it follows that 
	$T_1, \dots, T_m$ are pairwise disjoint measurable subsets of $I$. 
	Further, as $V(R_{\mathbf{w}}) = \bigcup_{i=1}^{m} T_i$, we obtain 
	\begin{equation*}
		\sum_{i=1}^{m} \int_{T_i} \lVert F(x) \rVert^2 d_W(x)\, \mathrm{d}x
		= \int_{V(R_{\mathbf{w}})} \lVert F(x) \rVert^2 d_W(x)\, \mathrm{d}x
		= X(\mathbf{w})
		\geq k - \frac{1}{4}.
	\end{equation*}
	Let $1 \leq i, j \leq m$ with $i \neq j$, 
	and $x \in T_i$ and $y \in T_j$ be arbitrary. 
	Then, $F(x) \neq \mathbf{0}$ and $F(y) \neq \mathbf{0}$, and hence, 
	$d_F(x,y) = \left\lVert \overline{F}(x) - \overline{F}(y) \right\rVert$.
	Note that $\overline{F}(x)$ and $\overline{F}(y)$ belong to the cubes 
	$\widetilde{C}_{\mathbf{n}_i, \mathbf{w}}$ and $\widetilde{C}_{\mathbf{n}_j, \mathbf{w}}$,
	respectively, and these cubes are at least $\frac{s}{4k^2}$ apart. 
	Thus, we conclude that 
	\[
		d_F(x,y) \geq \frac{s}{4k^2} = \frac{1}{4\sqrt{5} k^3}.
	\]
	Now fix any $1 \leq i \leq m$, and let $\mathbf{v} \in R_i$ be arbitrary.
	As $\diam(R_i) \leq s \sqrt{k} = \frac{1}{\sqrt{5k}}$, for any $x \in T_i$,	
	we have $\left\lVert \overline{F}(x) - \mathbf{v} \right\rVert^2 \leq \frac{1}{5k}$, 
	that is, $\left\langle \overline{F}(x),\mathbf{v} \right\rangle \geq 1 - \frac{1}{10k}$. 
	Thus, the inequality 
	\[
		\langle F(x),\mathbf{v} \rangle^2 
		\geq \left( 1 - \frac{1}{10k} \right)^2 \lVert F(x) \rVert^2.
	\]
	holds for all $x \in T_i$. This implies
	 \[
		\int_{T_i} \langle F(x),\mathbf{v} \rangle^2 d_W(x)\, \mathrm{d}x
		\geq \left( 1 - \frac{1}{10k} \right)^2 
		\int_{T_i} \lVert F(x) \rVert^2 d_W(x)\, \mathrm{d}x,
	\] 
	and then, using \cref{lem:wellspread}, we get
	\[
		\int_{T_i} \lVert F(x) \rVert^2 d_W(x)\, \mathrm{d}x
		\leq \frac{1}{\left( 1 - \frac{1}{10k} \right)^2} 
		\leq \frac{1}{1 - \frac{1}{5k}}
		\leq \frac{5k}{5k-1} 
		\leq 1 + \frac{1}{5k-1} 
		\leq 1 + \frac{1}{4k},
	\]
	as desired.
\end{proof}

Given a measurable subset $A$ of $I$, the integral 
\[
	\int_{A} \lVert F(x) \rVert^2 d_W(x)\, \mathrm{d}x
\]	
is called the \emph{mass} of the set $A$, 
and given a finite set $\mathcal{A}$ of pairwise disjoint measurable subsets of $I$, 
its \emph{total mass} is defined to be the sum of the masses of its elements.

\begin{lemma} \label{lem:large-separated}
	Let $\{f_1, \dots, f_k\}$ be an orthonormal subset of $L^2(I,\nu)$.
	Then there exist pairwise disjoint measurable subsets $A_1, \dots, A_k$ of $I$ 
	such that the following conditions hold. 
	\begin{enumerate}[(a)]	
		\item $\int_{A_i} \lVert F(x) \rVert^2 d_W(x)\, \mathrm{d}x
		\geq \frac{1}{2}$ for all $1 \leq i \leq k$. 
		\item For any $1 \leq i,j \leq k$ with $i \neq j$, 
		if $x \in A_i$ and $y \in A_j$, then  
		$d_F(x,y) \geq \frac{1}{4 \sqrt{5} k^3}$.
	\end{enumerate}
\end{lemma}

\begin{proof}
	Using \cref{prop:small-separated}, choose a positive integer $m$, 
	and pairwise disjoint measurable subsets $T_1, \dots, T_m$ of $I$ 
	satisfying the conditions $(a)$, $(b)$ and $(c)$ of \cref{prop:small-separated}. 
	We begin with the collection $\mathcal{A}_0 = \{T_1, \dots, T_m\}$ 
	and run the following process. 
	For $n \geq 1$, if the collection $\mathcal{A}_{n-1}$ contains 
	at least two elements each having mass less than $\frac{1}{2}$, 
	then pick any two of them and replace them with their union, 
	and denote the obtained collection by $\mathcal{A}_n$.
	Otherwise, stop the process.  
	As $\mathcal{A}_0$ is a finite collection, the process eventually stops, 
	say at the $r$-th step, for some $r \geq 0$. 
	Note that the collection $\mathcal{A}_r$ contains 
	at most one set having mass less than $\frac{1}{2}$, 
	and it contains at least one set having mass at least $\frac{1}{2}$, 
	since for any $i \geq 0$, the total mass of the collection $\mathcal{A}_i$ 
	is same as that of $\mathcal{A}_0$, 
	and the total mass of $\mathcal{A}_0$ is at least $\frac{3}{4}$. 
	Let $A_1, \dots, A_t$ denote the elements of $\mathcal{A}_r$, 
	each having mass at least $\frac{1}{2}$, for some $t \geq 1$.
	Observe that the mass of any set in the collection 
	$\{A_1, \dots, A_t\} \setminus \mathcal{A}_0$ is less than $1$, 
	as it is a union of two sets each having mass less than $\frac{1}{2}$.
	Also, the mass of any element of $\mathcal{A}_0$ is at most $1 + \frac{1}{4k}$.
	Hence, the total mass of $\mathcal{A}_0$ is at most 
	$\frac{1}{2} + t \left( 1 + \frac{1}{4k} \right)$. 
	Since we know that the total mass of $\mathcal{A}_0$ is at least $k - \frac{1}{4}$, 
	it follows that
	\[
		k - \frac{1}{4} \leq \frac{1}{2} + t \left( 1 + \frac{1}{4k} \right),
	\]	
	that is,
	\[
		t \geq \frac{k - \frac{3}{4}}{1 + \frac{1}{4k}},
	\] 
	and hence, $t > k - 1$. As $t$ is an integer, this implies $t \geq k$.
	Now observe that the sets $A_1, \dots, A_k$ are pairwise disjoint 
	measurable subsets of $I$ having the desired properties.
\end{proof}

\begin{prop} \label{prop:dis-supp-O(k7)}
	Given any orthonormal subset $\{f_1, \dots, f_k\}$ of $L^2(I,\nu)$,
	there exist disjointly supported nonzero functions $g_1, \dots, g_k \in L^2(I,\nu)$ 
	such that for all $1 \leq i \leq k$, the inequality
	\[
		R_{\Delta_W}(g_i) \leq O(k^7) \max_{1 \leq j \leq k} R_{\Delta_W}(f_j)
	\] 
	holds.
\end{prop}

\begin{proof}
	Let $\{f_1, \dots, f_k\}$ be an orthonormal subset of $L^2(I,\nu)$. 
	Choose pairwise disjoint measurable subsets $A_1, \dots, A_k$ of $I$ such that 
	\[
		\int_{A_i} \lVert F(x) \rVert^2 d_W(x)\, \mathrm{d}x \geq \frac{1}{2},
	\]	
	for all $1 \leq i \leq k$, 
	and for $i \neq j$, if $x \in A_i$ and $y \in A_j$, 
	then $d_F(x,y) \geq \frac{1}{4 \sqrt{5} k^3}$, 
	as guaranteed by \cref{lem:large-separated}. 
	Let $\delta = \frac{1}{4 \sqrt{5} k^3}$. 
	For each $1 \leq i \leq k$, define a function $\tau_i \colon I \to \mathbb{R}$ by 
	\[
		\tau_i(x) = \max \left\{ 0, 1 - \frac{2}{\delta} d_F(x,A_i) \right\},
	\]	
	for all $x \in I$. This is a measurable function thanks to \cref{lem:dist-to-set-meable}.
	Since the function $F$ is measurable, the function $g_i \colon I \to \mathbb{R}$,
	defined by $g_i(x) = \tau_i(x) \lVert F(x) \rVert$ for all $x \in I$, 
	is also measurable for every $1 \leq i \leq k$. 
	Moreover, it belongs to $L^2(I,\nu)$, as $\tau_i \leq 1$ 
	and the function $x \mapsto \lVert F(x) \rVert$ lies in $L^2(I,\nu)$. 
	For each $1 \leq i \leq k$, note that 
	\begin{align} \label{ineq:norm-gi}
		\lVert g_i \rVert_v^2 
		& = 
		\int_{I} \tau_i(x)^2 \lVert F(x) \rVert^2 d_W(x)\, \mathrm{d}x \nonumber
		\\
		& \geq 
		\int_{A_i} \tau_i(x)^2 \lVert F(x) \rVert^2 d_W(x)\, \mathrm{d}x \nonumber
		\\
		& =
		\int_{A_i} \lVert F(x) \rVert^2 d_W(x)\, \mathrm{d}x \nonumber
		\\
		& \geq 
		\frac{1}{2},
	\end{align}
	and hence, the function $g_i$ is nonzero. 
	Further, $g_1, \dots, g_k$ are disjointly supported functions. 
	Indeed, if $i, j \in \{1, \dots, k\}$ and $x \in I$ are such that 
	$g_i(x)$ and $g_j(x)$ are nonzero, then $\tau_i(x)$ and $\tau_j(x)$ 
	are nonzero. 
	Thus, we have $d_F(x,A_i) < \frac{\delta}{2}$ and $d_F(x,A_j) < \frac{\delta}{2}$, 
	and as a consequence, there exist $y_i \in A_i$ and $y_j \in A_j$ such that 
	$d_F(x,y_i) < \frac{\delta}{2}$ and $d_F(x,y_j) < \frac{\delta}{2}$. 
	Then, the triangle inequality for $d_F$ implies that $d_F(y_i,y_j) < \delta$, 
	which forces $i$ and  $j$ to be equal.

	Fix any $1 \leq i \leq k$. 
	We claim that for all $x, y \in I$, the inequality
	$\lvert g_i(x) - g_i(y) \rvert \leq \lVert F(x) - F(y) \rVert 
	\left( 1 + \frac{4}{\delta} \right)$ holds. 
	If $F(x) = \mathbf{0}$ or $F(y) = \mathbf{0}$, then using $\tau_i \leq 1$,
	we obtain the claimed inequality. 
	Assume that $F(x) \neq \mathbf{0}$ and $F(y) \neq \mathbf{0}$. 
	If $d_F(x,A_i) \geq \frac{\delta}{2}$ and $d_F(y,A_i) \geq \frac{\delta}{2}$, 
	or $d_F(x,A_i) < \frac{\delta}{2}$ and $d_F(y,A_i) < \frac{\delta}{2}$, then, 
	we have $\lvert \tau_i(x) - \tau_i(y) \rvert 
	\leq \frac{2}{\delta} \lvert d_F(x,A_i) - d_F(y,A_i) \rvert$.
	Moreover, if $d_F(x,A_i) < \frac{\delta}{2}$ and 
	$d_F(y,A_i) \geq \frac{\delta}{2}$, then it follows that
	\[
		\lvert \tau_i(x) - \tau_i(y) \rvert = 1 - \frac{2}{\delta} d_F(x,A_i)
		\leq \frac{2}{\delta} d_F(y,A_i) - \frac{2}{\delta} d_F(x,A_i),
	\] 
	and similarly, if $d_F(x,A_i) \geq \frac{\delta}{2}$ and 
	$d_F(y,A_i) < \frac{\delta}{2}$, then the inequality 
	$\lvert \tau_i(x) - \tau_i(y) \rvert 
	\leq \frac{2}{\delta} d_F(x,A_i) - \frac{2}{\delta} d_F(y,A_i)$ follows.
	Hence, in any of the cases, we obtain
	\[
		\lvert \tau_i(x) - \tau_i(y) \rvert 
		\leq \frac{2}{\delta} \lvert d_F(x,A_i) - d_F(y,A_i) \rvert 
		\leq \frac{2}{\delta} d_F(x,y).
	\] 
	Now note that 
	\begin{align*} 
		\lvert g_i(x) - g_i(y) \rvert 
		& = 
		\left\lvert \tau_i(x) \lVert F(x) \rVert 
		- \tau_i(y) \lVert F(y) \rVert \right\rvert
		\\
		& \leq
		\left\lvert \tau_i(x) \lVert F(x) \rVert 
		- \tau_i(x) \lVert F(y) \rVert \right\rvert
		+ \left\lvert \tau_i(x) \lVert F(y) \rVert 
		- \tau_i(y) \lVert F(y) \rVert \right\rvert 
		\\
		& =
		\tau_i(x) \left\lvert \lVert F(x) \rVert - \lVert F(y) \rVert \right\rvert
		+ \lVert F(y) \rVert \lvert \tau_i(x) - \tau_i(y) \rvert
		\\
		& \leq 
		\lVert F(x) - F(y) \rVert + \frac{2}{\delta} \lVert F(y) \rVert d_F(x,y),
	\end{align*}
	and that 
	\begin{align*}
		\lVert F(y) \rVert d_F(x,y) 
		& = 
		\lVert F(y) \rVert \left\lVert \frac{F(x)}{\lVert F(x) \rVert} 
		- \frac{F(y)}{\lVert F(y) \rVert} \right\rVert
		\\
		& =
		\left\lVert \frac{\lVert F(y) \rVert}{\lVert F(x) \rVert} 
		F(x) - F(y) \right\rVert
		\\
		& \leq
		\left\lVert \frac{\lVert F(y) \rVert}{\lVert F(x) \rVert} 
		F(x) - F(x) \right\rVert + \lVert F(x) - F(y) \rVert
		\\
		& =
		\left\lvert \frac{\lVert F(y) \rVert}{\lVert F(x) \rVert} - 1 \right\rvert
		\lVert F(x) \rVert + \lVert F(x) - F(y) \rVert
		\\
		& =
		\left\lvert \lVert F(x) \rVert - \lVert F(y) \rVert \right\rvert
		+ \lVert F(x) - F(y) \rVert
		\\
		& \leq
		2 \lVert F(x) - F(y) \rVert.
	\end{align*}
	Thus, we obtain
	\begin{equation} \label{ineq:gix-giy}
		\lvert g_i(x) - g_i(y) \rvert 
		\leq \lVert F(x) - F(y) \rVert \left( 1 + \frac{4}{\delta} \right).
	\end{equation}
	Then, \cref{eq:num-of-R_DeltaW} and \cref{ineq:gix-giy} together imply that
	\begin{align*}
		\langle \Delta_W g_i, g_i \rangle_v 
		& =
		\frac{1}{2} \int_{I} \int_{I} (g_i(x) - g_i(y))^2 W(x,y)\, 
		\mathrm{d}x\, \mathrm{d}y
		\\
		& \leq 
		\frac{1}{2} \left( 1 + \frac{4}{\delta} \right)^2 \int_{I} \int_{I} 
		\lVert F(x) - F(y) \rVert^2 W(x,y)\, \mathrm{d}x\, \mathrm{d}y
		\\
		& =
		\frac{1}{2} \left( 1 + \frac{4}{\delta} \right)^2 \sum_{j=1}^{k}
		\int_{I} \int_{I} (f_j(x) - f_j(y))^2 W(x,y)\, \mathrm{d}x\, \mathrm{d}y
		\\
		& =
		\left( 1 + \frac{4}{\delta} \right)^2 \sum_{j=1}^{k} R_{\Delta_W}(f_j)
		\tag{using $\lVert f_j \rVert_v = 1$ for each $1 \leq j \leq k$}
		\\
		& \leq
		\frac{25 k}{\delta^2} \max_{1 \leq j \leq k} R_{\Delta_W}(f_j) 
		\tag{using $\delta \leq 1$}
		\\
		& =
		2000 k^7 \max_{1 \leq j \leq k} R_{\Delta_W}(f_j).
	\end{align*}
	Combining this with \cref{ineq:norm-gi} gives us the inequality
	\[
		R_{\Delta_W}(g_i) \leq 4000 k^7 \max_{1 \leq j \leq k} R_{\Delta_W}(f_j),
	\] 
	for all $1 \leq i \leq k$.
\end{proof}

\begin{lemma} 
	\label{lem:supp-subset-small-exp}
	Let $g$ be a nonzero element of $L^2(I,\nu)$. 
	Then there exists a measurable subset $A$ of its support 
	such that $A$ has positive measure,
	and 
	satisfies the inequality $h_W(A) \leq \sqrt{2 R_{\Delta_W}(g)}$. 
\end{lemma}

\begin{proof}
	For every $t \geq 0$, let $A_t$ denote the set $\{x \in I \mid g(x)^2 > t\}$.
	Observe that $A_t$ is a measurable subset of the support of $g$ for each $t$.
	We claim that there is a nonnegative real number $t_0$ such that 
	the set $A_{t_0}$ has positive measure and it satisfies the inequality 
	$h_W(A_{t_0}) \leq \sqrt{2 R_{\Delta_W}(g)}$.

	We have
	\begin{align*}
		& \hspace{0.5cm} 
		\int_{0}^{\infty} \eta(A_t \times A_t^c) \, \mathrm{d}t
		\\
		& =
		\int_{0}^{\infty} \int_{I} \int_{I} 1_{A_t}(x) 1_{A_t^c}(y) W(x,y)\, 
		\mathrm{d}x\, \mathrm{d}y\, \mathrm{d}t
		\\
		& =
		\int_{\substack{(x,y) \in I \times I \\ g(y)^2 < g(x)^2}} 
		\left( \int_{g(y)^2}^{g(x)^2} 1\, \mathrm{d}t \right)
		W(x,y)\, \mathrm{d}x\, \mathrm{d}y
		\\
		& =
		\int_{\substack{(x,y) \in I \times I \\ g(y)^2 < g(x)^2}}
		(g(x)^2 - g(y)^2) W(x,y)\, \mathrm{d}x\, \mathrm{d}y
		\\
		& =
		\frac{1}{2} \int_{I} \int_{I} \lvert g(x)^2 - g(y)^2 \rvert W(x,y)\, 
		\mathrm{d}x\, \mathrm{d}y
		\\
		& =
		\frac{1}{2} \int_{I} \int_{I} \lvert g(x) - g(y) \rvert 
		\lvert g(x) + g(y) \rvert W(x,y)\, \mathrm{d}x\, \mathrm{d}y
		\\
		& \leq
		\frac{1}{2} \left( \int_{I} \int_{I} (g(x) - g(y))^2 W(x,y)\, 
		\mathrm{d}x\, \mathrm{d}y \right)^{\frac{1}{2}} 
		\left( \int_{I} \int_{I} (g(x) + g(y))^2 W(x,y)\, 
		\mathrm{d}x\, \mathrm{d}y \right)^{\frac{1}{2}}
		\tag{using the Cauchy--Schwarz inequality}
		\\
		& \leq
		\frac{1}{2} \left( \int_{I} \int_{I} (g(x) - g(y))^2 W(x,y)\, 
		\mathrm{d}x\, \mathrm{d}y \right)^{\frac{1}{2}} 
		\left( \int_{I} \int_{I} (2g(x)^2 + 2g(y)^2) W(x,y)\,
		\mathrm{d}x\, \mathrm{d}y \right)^{\frac{1}{2}}
		\\
		& =
		\frac{1}{2} \left( \int_{I} \int_{I} (g(x) - g(y))^2 W(x,y)\, 
		\mathrm{d}x\, \mathrm{d}y \right)^{\frac{1}{2}} 
		\left( 4 \int_{I} \int_{I} g(x)^2 W(x,y)\, 
		\mathrm{d}x\, \mathrm{d}y \right)^{\frac{1}{2}} 
		\\
		& =
		\left( \int_{I} \int_{I} (g(x) - g(y))^2 W(x,y)\, 
		\mathrm{d}x\, \mathrm{d}y \right)^{\frac{1}{2}} 
		\left( \int_{I} \int_{I} g(x)^2 W(x,y)\, 
		\mathrm{d}x\, \mathrm{d}y \right)^{\frac{1}{2}}, 
	\end{align*}
	and
	\begin{align*}
		\int_{0}^{\infty} \nu(A_t)\, \mathrm{d}t
		& =
		\int_{0}^{\infty} \int_{I} \int_{I} 1_{A_t}(x) W(x,y)\, 
		\mathrm{d}x\, \mathrm{d}y\, \mathrm{d}t
		\\
		& =
		\int_{I} \int_{I} \left( \int_{0}^{g(x)^2} 1\, \mathrm{d}t \right) W(x,y)\,
		\mathrm{d}x\, \mathrm{d}y
		\\
		& =
		\int_{I} \int_{I} g(x)^2 W(x,y)\, \mathrm{d}x\, \mathrm{d}y,
	\end{align*}
	which is equal to $\lVert g \rVert_v^2$, and it is positive, as $g$ is nonzero.
	Hence, it follows that
	\[
		\frac{\int_{0}^{\infty} \eta(A_t \times A_t^c) \, \mathrm{d}t}
		{\int_{0}^{\infty} \nu(A_t)\, \mathrm{d}t}
		\leq \sqrt{\frac{\int_{I} \int_{I} (g(x) - g(y))^2 W(x,y)\, 
		\mathrm{d}x\, \mathrm{d}y}
		{\int_{I} \int_{I} g(x)^2 W(x,y)\, \mathrm{d}x\, \mathrm{d}y}} 	
		= \sqrt{2 R_{\Delta_W}(g)}. 
	\] 
	Let $S$ denote the support of the function $t \mapsto \nu(A_t)$ defined on $[0,\infty)$.
	Since $\int_{0}^{\infty} \nu(A_t)\, \mathrm{d}t$ is positive, 
	it follows that the set $S$ has positive measure. 
	Then, using the pigeonhole principle 
	(see \cite[Lemma 4.3]{Mugdha-duaCBGraphons25} for instance),
	we conclude that there is an element $t_0$ of $S$ such that 
	 \[
		\frac{\eta(A_{t_0} \times A_{t_0}^c)}{\nu(A_{t_0})}
		\leq \frac{\int_S \eta(A_t \times A_t^c) \, \mathrm{d}t}
		{\int_S \nu(A_t)\, \mathrm{d}t}
		\leq \frac{\int_{0}^{\infty} \eta(A_t \times A_t^c) \, \mathrm{d}t}
		{\int_{0}^{\infty} \nu(A_t)\, \mathrm{d}t}.
	\] 
	Note that the set $A_{t_0}$ has positive measure and that the inequality 
	$h_W(A_{t_0}) \leq \sqrt{2 R_{\Delta_W}(g)}$ holds.
\end{proof}

\section*{Acknowledgements}

The author is grateful to Jyoti Prakash Saha for suggesting the problem, 
and for his valuable comments on the earlier versions of this article.
The author would like to thank Angshuman Bhattacharya and Somnath Pradhan 
for having useful discussions. 
The author also acknowledges the fellowship 
from the University Grants Commission with reference number 231620172128.


\end{document}